\documentclass{article}
\usepackage{amsmath, amssymb, amsthm, graphics, setspace}
\usepackage{mathrsfs}
\usepackage{multirow}
\usepackage[colorlinks, citecolor=blue, anchorcolor=blue, linkcolor=blue, urlcolor=blue]{hyperref}

\makeatletter \@addtoreset{equation}{section}

\newtheorem{theorem}{Theorem}[section]
\newtheorem{lemma}[theorem]{Lemma}
\newtheorem{prop}[theorem]{Proposition}

\newtheorem{remark}[theorem]{Remark}

\begin{document}
\begin{center}
{\large \bf On the positive zeros of generalized Narayana polynomials related to the Boros-Moll polynomials}
\end{center}
\begin{center}
James Jing Yu Zhao\\[6pt]

School of Mathematics,
Tianjin University, \\
Tianjin 300350, P.R. China\\[8pt]

Email: {\tt jjyzhao@tju.edu.cn}
\end{center}

\noindent\textbf{Abstract.}
The generalized Narayana polynomials $N_{n,m}(x)$ arose from the study of infinite log-concavity of the Boros-Moll polynomials. The real-rootedness of $N_{n,m}(x)$ had been proved by Chen, Yang and Zhang. They also showed that when $n\geq m+2$, each of the generalized Narayana polynomials has one and only one positive zero and $m$ negative zeros, where the negative zeros of $N_{n,m}(x)$ and $N_{n+1,m+1}(x)$ have interlacing relations.
In this paper, we study the properties of the positive zeros of $N_{n,m}(x)$ for $n\geq m+2$. We first obtain a new recurrence relation for the generalized Narayana polynomials. Based on this recurrence relation, we prove upper and lower bounds for the positive zeros of $N_{n,m}(x)$. Moreover, the monotonicity of the positive zeros of $N_{n,m}(x)$ are also proved by using the new recurrence relation.

\noindent \emph{AMS Classification 2020:} 05A10, 11B83, 26C10


\section{Introduction}

Let $n>k\geq 0$ be integers. The classical Narayana number $N(n,k)$, named after T.V. Narayana \cite{Narayana}, is given by
$$N(n,k)=\frac{1}{n}{n\choose k}{n\choose k+1},$$
which appears in OEIS as A001263 in \cite{Sloane}. It is well known that the Narayana numbers refine the Catalan numbers $C_n=\frac{1}{n+1}{2n\choose n}$ since
$$\sum_{k=0}^{n-1} N(n,k)=C_n.$$
For more information on Catalan numbers, see \cite{Catalan, stanley}.
The generating polynomials of $N(n,k)$, namely
$$
\sum_{k=0}^{n-1} N(n,k) x^k,
$$
are called the Narayana polynomials.
The Narayana numbers and the Narayana polynomials have attracted a lot of attention,
and been extensively studied in relation to algebraic combinatorics \cite{Williams2005, Armstrong, Avaletal, ARR2015, MMY2019}, number theory \cite{GuoJiang2017, SaganTirrell2020}, probability and statistics \cite{Sulanke02, Branden2004, Avaletal, Ful-Rol, C-Y-Z-2021}, geometry \cite{ARR2015} and especially to enumerative combinatorics \cite{Sulanke04, Kamioka, Avaletal,  CYZ1, CYZ2, WangYang, SaganTirrell2020} by the mathematical community.

It is well known that the Narayana polynomials have only real zeros, see \cite{LiuWang}. In 2018, Chen, Yang and Zhang \cite{CYZ2} studied a generalization of the Narayana polynomials in the following form
\begin{align}\label{eq-geNaPoCYZ}
N_{n,m}(x)=\sum\limits_{k=0}^n \left({n\choose k}{m\choose k}-{n\choose k+1}{m\choose k-1}\right) x^k,
\end{align}
where $m$ and $n$ are nonnegative integers.
Clearly, when $n=m+1$, the generalized Narayana polynomials $N_{n,m}(x)$ reduce to the classical Narayana polynomials, say,
\begin{align*}
N_{m+1,m}(x)=\frac{1}{m+1}\sum_{k=0}^{m}{m+1\choose k}{m+1\choose k+1}x^k=\sum_{k=0}^{m}N(m+1,k)x^k.
\end{align*}

It should be mentioned that the generalized Narayana polynomials $N_{n,m}(x)$ arose in the study of infinite log-concavity of the Boros-Moll polynomials, which were first introduced by Boros and Moll \cite{Boros-Moll} while studying a quartic integral. The generalized Narayana polynomials $N_{n,m}(x)$ appear to have some interesting properties on their zeros. For instance, Chen, Yang and Zhang \cite{CYZ2} had showed the real-rootedness of $N_{n,m}(x)$ for nonnegative integers $m$ and $n$. Moreover, the real zeros of $N_{n,m}(x)$ have interlacing relations for $n\leq m+1$. Specifically, for $n\geq m+2$, they gave the following result on positive zeros.
\begin{theorem}\cite[Theorem 3.4]{CYZ2}\label{thm-unipzo}
For any $m\geq 0$ and $n\geq m$, the polynomial $N_{n,m}(x)$ has only real zeros.
If $n\geq m+2$, then $N_{n,m}(x)$ has one and only one positive zero.
\end{theorem}
The real-rootedness of $N_{n,m}(x)$ was proved by a criterion for determining whether two polynomials have interlaced zeros established by Liu and Wang \cite[Theorem 2.3]{LiuWang} and the theory of P\'{o}lya frequency sequences \cite{Wang-Yeh}.
The existence and uniqueness of the positive zero of each $N_{n,m}(x)$ for $n\geq m+2$ can be obtained easily by using the well known Intermediate Value Theorem together with Descartes's Rule (see \cite{Curtiss}).
Chen et al. \cite[Theorem 3.4]{CYZ2} also proved that $N_{n,m}(x)$ has $m$ negative zeros for $n\geq m+2$ with $m\geq 0$, and the negative zeros of $N_{n,m}(x)$ and $N_{n+1,m+1}(x)$ have interlacing relations.
Many well known functions (or polynomials) have interlacing properties for their zeros. For instance, Cho and Chung \cite{Cho-Chung} had proved that the positive zeros of $\nu$-parameter families of Bessel functions are simultaneously interlaced under certain conditions.
Although the real-rootedness of $N_{n,m}(x)$ and the interlace feature of their negative zeros had been deeply studied, the properties of the positive zeros of $N_{n,m}(x)$ were still left unknown.

This paper mainly concerns with the analytic properties of the positive zeros of the generalized Narayana polynomials $N_{n,m}(x)$.
Note that the case of $n=m+2$ is trivial since $N_{m+2,m}(1)=0$ by the Chu-Vandermonde convolution (see \cite{Gould} or \cite[\S 5.1]{GKP}). For $n\geq m+3$ with $m\geq 0$, we give upper and lower bounds for the the positive zeros of $N_{n,m}(x)$. Furthermore, we also show monotonicity of the positive zeros of $N_{n,m}(x)$. These two main results are proved by using mathematical induction together with a new three-term recurrence relation of the generalized Narayana polynomials.

This paper is organized as follows. In Section \ref{S-R}, we give a new three-term recurrence relation of $N_{n,m}(x)$  and show a proof by hands. An alternative proof by symbolic method established by Chen, Hou and Mu \cite{CHM} is also mentioned. In Section \ref{S-3}, we prove the first main result of this paper, the upper and lower bounds for the positive zeros of $N_{n,m}(x)$. The second main result of this paper, say, the monotonicity of the positive zeros of $N_{n,m}(x)$, are stated in Section \ref{S-M}.

\section{Recurrence relation}\label{S-R}
In this section we show a new three-term recurrence relation of the generalized Narayana polynomials $N_{n,m}(x)$ defined in \eqref{eq-geNaPoCYZ}. This recurrence relation will be used to prove the main results of this paper, the bounds and the monotonicity of the positive zeros of $N_{n,m}(x)$ for $n\geq m+3$ and $m\geq 0$.

The main result of this seciotn is as follows.
\begin{theorem}\label{thm-main-rec-gNyp}
For any integers $m\geq 0$ and $n\geq 1$, we have
\begin{align}\label{eq-rec-gelNaya-m}
c_{n,m}(x)N_{n,m+1}(x)=a_{n,m}(x) N_{n,m}(x)+b_{n,m}(x) N_{n-1,m}(x),
\end{align}
where
\begin{equation}\label{eq-abcnmx}
\left\{
\begin{aligned}
&a_{n,m}(x)=(m+2-n)(m^2-n^2+4m+3)x-2n,\\[5pt]
&b_{n,m}(x)=n[(m+2-n)(m+1-n)x-2](x-1),\\[5pt]
&c_{n,m}(x)=(m+3)(m+2-n)(m+1-n)x.
\end{aligned}
\right.
\end{equation}
\end{theorem}
\begin{proof}
First, let us calculate the right-hand side of \eqref{eq-rec-gelNaya-m}.
For convenience, rewrite
$$
a_{n,m}(x)=Ax-2n\quad {\rm and}\quad  b_{n,m}(x)=Bx^2-Cx+2n,
$$
where
\begin{align*}
&A=(m+2-n)(m^2-n^2+4m+3),\\
&B=n(m+2-n)(m+1-n),\\
&C=B+2n.
\end{align*}
Then the right-hand side of \eqref{eq-rec-gelNaya-m} can be expressed as a sum of three parts, that is,
\begin{align}
&\ a_{n,m}(x)N_{n,m}(x)+b_{n,m}(x)N_{n-1,m}(x)\nonumber\\
=&\ (Ax-2n) N_{n,m}(x)+(Bx^2-Cx+2n)N_{n-1,m}(x)\nonumber\\
=&\ 2n\left(N_{n-1,m}(x)-N_{n,m}(x)\right)+\left(Ax N_{n,m}(x)-Cx N_{n-1,m}(x)\right)
+Bx^2 N_{n-1,m}(x).\label{eq-2.3}
\end{align}

Observe that the coefficients of $x^k$ in the summand of \eqref{eq-geNaPoCYZ} can be rewritten as
\begin{align*}
 {n\choose k}{m\choose k}-{n\choose k+1}{m\choose k-1}
={n+1\choose k+1}{m+1\choose k}\frac{(m-n)k+m+1}{(n+1)(m+1)}.
\end{align*}
So it follows that
\begin{align}
2n\left(N_{n-1,m}(x)-N_{n,m}(x)\right)
&=2n\sum_{k=0}^n {m\choose k-1}{n\choose k}\frac{n-m-1}{n}x^k\nonumber\\
&=2nx\sum_{k=1}^n {m\choose k-1}{n\choose k}\frac{n-m-1}{n}x^{k-1}\nonumber\\
&=x\sum_{k=0}^{n} 2(n-m-1) {m\choose k}{n\choose k+1}x^k. \label{eq-RHS-1}
\end{align}
In order to calculate the second part of \eqref{eq-2.3}, set $A=C+D$, where $D=(m+1-n)(m^2-mn+5m-n+6)$.
Then we have
\begin{align}
Ax N_{n,m}(x)-Cx N_{n-1,m}(x)
&=(C+D)x N_{n,m}(x)-Cx N_{n-1,m}(x)\nonumber\\
&=Cx\left(N_{n,m}(x)-N_{n-1,m}(x)\right)+Dx N_{n,m}(x)\nonumber\\
&=x \sum_{k=0}^{n} C{m\choose k-1}{n\choose k}\frac{m+1-n}{n}x^k\nonumber\\
&\qquad +x \sum_{k=0}^{n} D{n+1\choose k+1}{m+1\choose k}\frac{(m-n)k+m+1}{(n+1)(m+1)}x^k.\label{eq-RHS-2}
\end{align}
For the last part of \eqref{eq-2.3}, we have
\begin{align}
Bx^2 N_{n-1,m}(x)
&=Bx \sum_{k=0}^{n-1} {n\choose k+1}{m+1\choose k}\frac{(m+1-n)k+m+1}{n(m+1)}x^{k+1}\nonumber\\
&=x \sum_{k=0}^{n} B {n\choose k}{m+1\choose k-1}\frac{(m+1-n)k+n}{n(m+1)}x^k. \label{eq-RHS-3}
\end{align}

Now we have that the right-hand side of \eqref{eq-rec-gelNaya-m} is equal to a sum of three parts as given in \eqref{eq-RHS-1}, \eqref{eq-RHS-2} and \eqref{eq-RHS-3}.
Clearly, the left-hand side of \eqref{eq-rec-gelNaya-m} is
$$
c_{n,m}(x)N_{n,m+1}(x)
=x\sum_{k=0}^n (m+3)(m+2-n)(m+1-n) {n+1\choose k+1}{m+2\choose k}\frac{(m+1-n)k+m+2}{(n+1)(m+2)}x^k.
$$
By comparing the coefficients of $x^k$ in the summands of \eqref{eq-RHS-1}, \eqref{eq-RHS-2}, \eqref{eq-RHS-3} and $c_{n,m}(x)N_{n,m+1}(x)$, we find that
\begin{align*}
&2(n-m-1){m\choose k}{n\choose k+1}+C{m\choose k-1}{n\choose k}\frac{m+1-n}{n}
+D{n+1\choose k+1}{m+1\choose k}\frac{(m-n)k+m+1}{(n+1)(m+1)}\nonumber\\[5pt]
&\quad +B{n\choose k}{m+1\choose k-1}\frac{(m+1-n)k+n}{n(m+1)}\nonumber\\[5pt]
=&\ (m+3)(m+2-n)(m+1-n) {n+1\choose k+1}{m+2\choose k}\frac{(m+1-n)k+m+2}{(n+1)(m+2)},
\end{align*}
where $B,C$ and $D$ are defined above.
Therefore it follows that \eqref{eq-rec-gelNaya-m} holds for all integers $m\geq 0$ and $n\geq 1$.
This completes the proof.
\end{proof}

Specifically, when $n\neq m+1$ and $n\neq m+2$, from \eqref{eq-rec-gelNaya-m} we have
\begin{align}\label{eq-rec-gelNaya-mv}
N_{n,m+1}(x)
=\frac{a_{n,m}(x)}{c_{n,m}(x)}N_{n,m}(x)+\frac{b_{n,m}(x)}{c_{n,m}(x)}N_{n-1,m}(x),
\end{align}
where $a_{n,m}(x)$, $b_{n,m}(x)$ and $c_{n,m}(x)$ are given by \eqref{eq-abcnmx}. In the remainder of this paper, we shall use the recurrence relation \eqref{eq-rec-gelNaya-mv} to prove our main results.

\begin{remark}
It should be mentioned that Theorem \ref{thm-main-rec-gNyp} can also be proved by using the extended Zeilberger algorithm, a symbolic method established by Chen, Hou and Mu \cite{CHM}. See \cite{C-Y-Z-2021} for example.
\end{remark}

\section{The bounds}\label{S-3}
The aim of this section is to prove the first main result of this paper, the lower and upper bounds of the positive zeros of $N_{n,m}(x)$.
\begin{theorem}\label{Thm-zero-bounds}
For $m\geq 0$ and $n\geq m+3$, denote by $r_{n,m}^{+}$ the positive zero of the generalized Narayana polynomial $N_{n,m}(x)$. Then we have that
\begin{align}\label{eq-zero-bounds}
\frac{2(n+1)}{(m+1-n)((m+2)^2-(n+1)^2-1)}
<r_{n,m}^{+}
\leq \frac{2}{(m-n)(m+1-n)},
\end{align}
where the equality holds only for $m=0$.
\end{theorem}

Before proving Theorem \ref{Thm-zero-bounds}, we first show the following two lemmas, which will be used in the proofs of our main results.
\begin{lemma}\label{lemma-mono-on-n}
Fixed $m\geq 0$ and $n\geq m+4$. Then for any $x>0$ we have
\begin{align*}
N_{n-1,m}(x)>N_{n,m}(x).
\end{align*}
\end{lemma}
\begin{proof}
Note that when $n\geq m+2$, the polynomial $N_{n,m}(x)$ has degree $m+1$. So by \eqref{eq-geNaPoCYZ}, we have
\begin{align*}
N_{n-1,m}(x)-N_{n,m}(x)
=&\,\sum_{k=0}^{m+1}\left[{n-1\choose k}{m\choose k}-{n-1\choose k+1}{m\choose k-1}
 -\left({n\choose k}{m\choose k}-{n\choose k+1}{m\choose k-1}\right)\right]x^k\\
=&\,\sum_{k=0}^{m+1}\left(-{n-1\choose k-1}{m\choose k}+{n-1\choose k}{m\choose k-1}\right)x^k\\
=&\,\sum_{k=0}^{m+1}{n\choose k}{m\choose k-1}\frac{n-m-1}{n}x^k.
\end{align*}
It follows that $N_{n-1,m}(x)-N_{n,m}(x)>0$ for $m\geq0$, $n\geq m+4$ and $x>0$.
This completes the proof.
\end{proof}

\begin{lemma}\label{lem-sign}
Given integers $m\geq 0$ and $n\geq m+2$, let $N_{n,m}(x)$ be defined as in \eqref{eq-geNaPoCYZ}, and $r_{n,m}^{+}$ be the positive zero of $N_{n,m}(x)$. Then for any $x>0$, we have
\begin{itemize}
\item[$(i)$] $N_{n,m}(x)>0$ if and only if $x<r_{n,m}^{+}$;
\item[$(ii)$] $N_{n,m}(x)<0$ if and only if $x>r_{n,m}^{+}$.
\end{itemize}
\end{lemma}
\begin{proof}
We first prove the sufficiency of $(i)$ and $(ii)$.
Suppose $0<x<r_{n,m}^{+}$. It is clear that $N_{n,m}(x)\neq 0$ by Theorem \ref{thm-unipzo}, because each $N_{n,m}(x)$ has one and only one positive zero when $m\geq 0$ and $n\geq m+2$.
By \eqref{eq-geNaPoCYZ}, each polynomial $N_{n,m}(x)$ is a continuous function with respect to $x$ in $(-\infty,+\infty)$.
If $N_{n,m}(x)<0$, then by the continuity of $N_{n,m}(x)$ there must exist at least one positive zero in $(0,x)$ since $N_{n,m}(0)=1>0$ by \eqref{eq-geNaPoCYZ}, which contradicts Theorem \ref{thm-unipzo}. It follows that $N_{n,m}(x)>0$.

Suppose $x>r_{n,m}^{+}$. Clearly, $N_{n,m}(x)\neq 0$ by Theorem \ref{thm-unipzo}.
Note that for $n\geq m+2$, the leading term of $N_{n,m}(x)$ is $-{n\choose m+2}x^{m+1}$ by \eqref{eq-geNaPoCYZ}. Hence $\lim_{x\rightarrow +\infty} N_{n,m}(x)=-\infty$ for $n\geq m+2$ and $m\geq 0$.
Assume $N_{n,m}(x)>0$, then by the continuity of $N_{n,m}(x)$, there is at least one positive zero in $(x,+\infty)$, a contradiction to Theorem \ref{thm-unipzo}.
This leads to that $N_{n,m}(x)<0$.

To show the necessity of $(i)$, suppose $N_{n,m}(x)>0$ for $x>0$. Clearly, $x\neq r_{n,m}^{+}$. If $x>r_{n,m}^{+}$, then by the sufficiency of $(ii)$ proved above, we have $N_{n,m}(x)<0$, a contradiction. It follows that $x<r_{n,m}^{+}$.
The necessity of $(ii)$ is obtained in a similar argument and the details are omitted.
This completes the proof.
\end{proof}

Now we are able to show the proof of the first main result of this paper.

\noindent{\it Proof of Theorem \ref{Thm-zero-bounds}.}
It is clear that $N_{n,m}(x)$ are polynomials with real coefficients and leading term $-{n\choose m+2}x^{m+1}$ for $m\geq 0$ and $n\geq m+3$. Hence $N_{n,m}(x)$ are continuous functions with respect to $x$ in $(-\infty,+\infty)$. Since the degree of each $N_{n,m}(x)$ is $m+1$ for $n\geq m+3$ and $m\geq 0$, we shall prove the bounds in \eqref{eq-zero-bounds} by mathematical induction on $m$.

For $m=0$, we have $N_{n,0}(x)=-{n\choose 2}x+1$ by \eqref{eq-geNaPoCYZ}. It is clear that
$$
\frac{2(n+1)}{(1-n)(3-(n+1)^2)}<r_{n,0}^{+}=\frac{2}{n(n-1)}
$$
for $n\geq 3$. Hence we have \eqref{eq-zero-bounds} holds for $m=0$ with $n\geq 3$.

For $m=1$, by \eqref{eq-geNaPoCYZ} we have $N_{n,1}(x)=-{n\choose 3}x^2-\frac{n(n-3)}{2}x+1$. Then
$$
N_{n,1}\left(\frac{2}{(n-1)(n-2)}\right)=-\frac{2(n-3)}{3(n-1)(n-2)}<0,\quad n\geq 4.
$$
It follows from Lemma \ref{lem-sign} that $r_{n,1}^{+}<\frac{2}{(n-1)(n-2)}$ for $n\geq 4$.
Moreover,
$$
N_{n,1}\left(\frac{2(n+1)}{(n-2)((n+1)^2-8)}\right)
=\frac{2(n-3)[(n-1)(2n^2+n-25)+24]}{3(n-2)(n^2+2n-7)^2}>0,
\quad n\geq 4.
$$
By Lemma \ref{lem-sign}, $r_{n,1}^{+}>\frac{2(n+1)}{(n-2)((n+1)^2-8)}$ for $n\geq 4$. Thus we have \eqref{eq-zero-bounds} holds true for $m=1$ with $n\geq 4$.

Now we have proved \eqref{eq-zero-bounds} for $m=0$ and $m=1$ with $n\geq m+3$. Next assume \eqref{eq-zero-bounds} holds for $m\geq 1$ and $n\geq m+3$. We aim to prove that \eqref{eq-zero-bounds} holds for $m+1$ and $n\geq m+4$.
That is, for $n\geq m+4$,
\begin{align}\label{eq-zero-bounds-m+1}
x_1<r_{n,m+1}^{+}<x_2,
\end{align}
where
$$
x_1=\frac{2(n+1)}{(m+2-n)((m+3)^2-(n+1)^2-1)}
\quad {\rm and}
\quad x_2=\frac{2}{(m+1-n)(m+2-n)}.
$$
Clearly $0<x_1<x_2$ for $n\geq m+4$ and $m\geq 0$. In order to prove \eqref{eq-zero-bounds-m+1}, it is sufficient to prove that for $n\geq m+4$,
$$
N_{n,m+1}(x_1)>0
\quad {\rm and} \quad
N_{n,m+1}(x_2)<0.
$$

We first prove $N_{n,m+1}(x_1)>0$. For this purpose, we use the recurrence relation \eqref{eq-rec-gelNaya-mv} to express $N_{n,m+1}(x_1)$ as
\begin{align}\label{eq-rec-proofm+1}
N_{n,m+1}(x_1)
=\frac{a_{n,m}(x_1)}{c_{n,m}(x_1)}N_{n,m}(x_1)
 +\frac{b_{n,m}(x_1)}{c_{n,m}(x_1)}N_{n-1,m}(x_1).
\end{align}
For $m\geq 0$ and $n\geq m+4$, it is clear that $0<x_1<1$ and hence $c_{n,m}(x_1)>0$ by \eqref{eq-abcnmx}.
By a simple calculation we get
$$
a_{n,m}(x_1)=-\frac{2(n-m-3)(n-m-1)}{(n+m+4)(n-m-2)+1}<0,\quad n\geq m+4.
$$
Observe that for $m\geq 0$ and $n\geq m+4$,
$$
(m+2-n)(m+1-n)x_1-2=-\frac{2(m+2)(n-m-3)}{(n+m+4)(n-m-2)+1}<0.
$$
So
$$
b_{n,m}(x_1)
=n[(m+2-n)(m+1-n)x_1-2](x_1-1)>0,\quad n\geq m+4.
$$
It follows that for $m\geq 0$ and $n\ge m+4$,
$$
\frac{a_{n,m}(x_1)}{c_{n,m}(x_1)}<0 \quad
{\rm and}
\quad
\frac{b_{n,m}(x_1)}{c_{n,m}(x_1)}>0.
$$

To determine the sign of $N_{n,m+1}(x_1)$ it remains to ensure the sign of $N_{n,m}(x_1)$ and $N_{n-1,m}(x_1)$.
We claim that $N_{n-1,m}(x_1)>0$ for $m\geq 0$ and $n\geq m+4$.
By hypothesis, we have
$$
\frac{2n}{(m+2-n)((m+2)^2-n^2-1)}<r_{n-1,m}^{+}.
$$
This leads to
\begin{align*}
&\ x_1-\frac{2n}{(m+2-n)((m+2)^2-n^2-1)}\\
=&\ -\frac{2(n-m-3)(n-m-1)}{(n-m-2)[(n+m+4)(n-m-2)+1][n^2-(m+1)(m+3)]}\\
<&\ 0
\end{align*}
for $m\geq 0$ and $n\geq m+4$.
So
$$
0<x_1<\frac{2n}{(m+2-n)((m+2)^2-n^2-1)}<r_{n-1,m}^{+}.
$$
It follows from Lemma \ref{lem-sign} that $N_{n-1,m}(x_1)>0$.
Thus for $m\geq 0$ and $n\geq m+4$, we have
\begin{align}\label{eq-boprm1}
\frac{b_{n,m}(x_1)}{c_{n,m}(x_1)}N_{n-1,m}(x_1)>0.
\end{align}

Now let us consider the sign of $N_{n,m}(x_1)$.
Notice that $N_{n,m}(x_1)$ can not be negative for all $n\geq m+4$.
For example, when $m=1$,
$$
N_{n,1}(x_1)=\frac{n^5-52n^4+123n^3+1018n^2-4666n+5292}{3(n-3)^2 (n^2+2n-14)^2}.
$$
This yields $N_{n,1}(x_1)=-61/49$ for $n=5$, and
$N_{n,1}(x_1)=48074/410346049$ for $n=50$.
So we shall discuss in two cases.

\nointerlineskip {\bf Case 1.}
$N_{n,m}(x_1)\leq 0$. In this case we have $\frac{a_{n,m}(x_1)}{c_{n,m}(x_1)}N_{n,m}(x_1)\geq 0$, and hence $N_{n,m+1}(x_1)>0$ for $m\geq 0$ and $n\geq m+4$ by \eqref{eq-rec-proofm+1} and \eqref{eq-boprm1}.

\nointerlineskip {\bf Case 2.}
$N_{n,m}(x_1)>0$. Since $x_1>0$, by Lemma \ref{lemma-mono-on-n} we have that $N_{n-1,m}(x_1)> N_{n,m}(x_1)$ for $m\geq 0$ and $n\geq m+4$. Then by \eqref{eq-rec-gelNaya-mv},
\begin{align*}
N_{n,m+1}(x_1)
&\,=\frac{a_{n,m}(x_1)}{c_{n,m}(x_1)}N_{n,m}(x_1)
  +\frac{b_{n,m}(x_1)}{c_{n,m}(x_1)}N_{n-1,m}(x_1)\\
&\,>\frac{a_{n,m}(x_1)+b_{n,m}(x_1)}{c_{n,m}(x_1)}N_{n,m}(x_1),
\end{align*}
where
\begin{align*}
a_{n,m}(x_1)+b_{n,m}(x_1)
&\,=\frac{2(n+1)(n-m-3)(n-m-1)[(n-m-3)(m+1)(n+m+4)-2]}{[(n+m+4)(n-m-2)+1]^2 (n-m-2)}>0
\end{align*}
for $m\geq 0$ and $n\geq m+4$. Therefore $N_{n,m+1}(x_1)>0$ for $m\geq 0$ and $n\geq m+4$.

Thus in both cases, it follows that $N_{n,m+1}(x_1)>0$
for all $m\geq 1$ and $n\geq m+4$.

It remains to prove the inequality $N_{n,m+1}(x_2)<0$ for $n\geq m+4$. By the recurrence relation \eqref{eq-rec-gelNaya-mv} we have
\begin{align*}
N_{n,m+1}(x_2)
&=\frac{a_{n,m}(x_2)}{c_{n,m}(x_2)}N_{n,m}(x_2)
 +\frac{b_{n,m}(x_2)}{c_{n,m}(x_2)}N_{n-1,m}(x_2).
\end{align*}
Observe that $(m+2-n)(m+1-n)x_2-2=0$ for $n\geq m+4$. So $b_{n,m}(x_2)=0$. Thus
\begin{align}\label{eq-rec-proofm1x2}
N_{n,m+1}(x_2)
=\frac{a_{n,m}(x_2)}{c_{n,m}(x_2)}N_{n,m}(x_2).
\end{align}

Let us determine the signs of $a_{n,m}(x_2)$, $c_{n,m}(x_2)$ and $N_{n,m}(x_2)$.
For $n\geq m+4$, it is clear that $0<x_2\leq 1/3$, and hence $c_{n,m}(x_2)>0$ by \eqref{eq-abcnmx}.
Notice that
\begin{align*}
a_{n,m}(x_2)=\frac{(m+2-n)(m^2-n^2+4m+3)2}{(m+1-n)(m+2-n)}-2n
=\frac{2(m+1)(n-m-3)}{n-m-1}>0
\end{align*}
for $n\geq m+4$.
By hypothesis, it follows that
$$
r_{n,m}^{+}<\frac{2}{(m-n)(m+1-n)}
<\frac{2}{(m+1-n)(m+2-n)}
=x_2<+\infty
$$
for $n\geq m+4$.
By Lemma \ref{lem-sign}, $N_{n,m}(x_2)<0$ for $n\geq m+4$.
Then by \eqref{eq-rec-proofm1x2} we have $N_{n,m+1}(x_2)<0$ for $n\geq m+4$.

Now we have proved \eqref{eq-zero-bounds-m+1} for $m\geq 1$ with $n\geq m+4$. By the inductive hypothesis, it follows that the inequalities in \eqref{eq-zero-bounds} hold for $m\geq 0$ and $n\geq m+3$.
By the procedure of the proof, it is clear that the equality in \eqref{eq-zero-bounds} holds only for $m=0$.
This completes the proof.
\qed

The following result is an immediate consequence of Theorem \ref{Thm-zero-bounds}. \begin{prop}
Let $m\geq 0$ and $n\geq m+3$. Denote by $r_{n,m}^{+}$ the positive zero of $N_{n,m}(x)$. Then $0<r_{n,m}^{+}\leq \frac{1}{3}$. In addition, for any fixed $m\geq 0$,
$$\lim_{n\rightarrow \infty} r_{n,m}^{+}=0.$$
\end{prop}

\section{The monotonicity}\label{S-M}
This section is devoted to the study of the monotonicity of the positive zeros of $N_{n,m}(x)$.
The second main result of this paper is as follows.

\begin{theorem}\label{thm-monocity}
Let $m\geq 0$ and $n\geq m+3$. Denote by $r_{n,m}^{+}$ the positive zero of $N_{n,m}(x)$, then we have
\begin{align}
r_{n+1,m}^{+}<r_{n,m}^{+}, \qquad r_{n+1,m+1}^{+}<r_{n,m}^{+},\qquad n\geq m+3,
\end{align}
and
\begin{align}
r_{n,m}^{+}<r_{n,m+1}^{+},& \qquad n\geq m+4.
\end{align}
\end{theorem}

\begin{proof}
Note that $N_{n,m}(x)$ are continuous functions with respect to $x$ in $(-\infty,+\infty)$.
By Lemma \ref{lemma-mono-on-n} we have $N_{n,m}(x)>N_{n+1,m}(x)$ for $m\geq 0$, $n\geq m+3$ and $x>0$. It follows from Theorem \ref{thm-unipzo} that \begin{align}\label{eq-mono-1}
r_{n+1,m}^{+}<r_{n,m}^{+}
\end{align}
for $m\geq 0$ and $n\geq m+3$.

Notice that $r_{n,m+1}^{+}<r_{n-1,m}^{+}$ for $n\geq m+4$ if and only if $r_{n+1,m+1}^{+}<r_{n,m}^{+}$ for $n\geq m+3$.
So it remains to prove
\begin{align}\label{eq-inequ-2}
r_{n,m}^{+}<r_{n,m+1}^{+}<r_{n-1,m}^{+}, \qquad n\geq m+4.
\end{align}
Next we shall use mathematical induction on $m$ to prove \eqref{eq-inequ-2}.

First for $m=0$, we aim to prove
$$
r_{n,0}^{+}<r_{n,1}^{+}<r_{n-1,0}^{+},\qquad n\geq 4.
$$
Clearly, $0<r_{n,0}^{+}<r_{n-1,0}^{+}$ for $n\geq 4$ by \eqref{eq-mono-1}. Thus by the continuity of $N_{n,m}(x)$ and Theorem \ref{thm-unipzo}, it suffices to show that for $n\geq 4$,
\begin{align}\label{eq-s-condi-m0}
N_{n,1}(r_{n,0}^{+})>0 \quad {\rm and}\quad N_{n,1}(r_{n-1,0}^{+})<0.
\end{align}

From \eqref{eq-geNaPoCYZ} we have $N_{n,0}(x)=-{n\choose 2}x+1$. So $r_{n,0}^{+}=2/n(n-1)$.
By the recurrence relation \eqref{eq-rec-gelNaya-mv},
\begin{align}\label{eq-rec-Nn10}
N_{n,1}(x)
=\frac{(2-n)(-n^2+3)x-2n}{3(2-n)(1-n)x}N_{n,0}(x)
 +\frac{n[(2-n)(1-n)x-2](x-1)}{3(2-n)(1-n)x}N_{n-1,0}(x).
\end{align}
Obviously, $N_{n,0}(r_{n,0}^{+})=0$. Then
\begin{align*}
N_{n,1}(r_{n,0}^{+})
=\frac{n[(2-n)(1-n)r_{n,0}^{+}-2](r_{n,0}^{+}-1)}
 {3(2-n)(1-n)r_{n,0}^{+}}N_{n-1,0}(r_{n,0}^{+}).
\end{align*}
By Lemma \ref{lem-sign} we have $N_{n-1,0}(r_{n,0}^{+})>0$ since $0<r_{n,0}^{+}<r_{n-1,0}^{+}$. Observe that $(2-n)(1-n)r_{n,0}^{+}-2=2(n-2)/n-2<0$, for $n\geq 4$. It is clear that $r_{n,0}^{+}-1<0$ and $3(2-n)(1-n)r_{n,0}^{+}>0$ for $n\geq 4$. Thus $N_{n,1}(r_{n,0}^{+})>0$ for $n\geq 4$.

In order to prove $N_{n,1}(r_{n-1,0}^{+})<0$, let us consider \eqref{eq-rec-Nn10} again. Clearly, $N_{n-1,0}(r_{n-1,0}^{+})=0$. Hence by \eqref{eq-rec-Nn10},
\begin{align*}
N_{n,1}(r_{n-1,0}^{+})
=\frac{(2-n)(-n^2+3)r_{n-1,0}^{+}-2n}{3(2-n)(1-n)r_{n-1,0}^{+}}
 N_{n,0}(r_{n-1,0}^{+}).
\end{align*}
By Lemma \ref{lemma-mono-on-n} we have $N_{n,0}(r_{n-1,0}^{+})<0$ since $r_{n,0}^{+}<r_{n-1,0}^{+}$. Notice that $(2-n)(-n^2+3)r_{n-1,0}^{+}-2n=2(n^2-3)/(n-1)-2n=2(n-3)/(n-1)>0$, for $n\geq 4$.
Clearly, $3(2-n)(1-n)r_{n-1,0}^{+}>0$ for $n\geq 4$. It follows that $N_{n,1}(r_{n-1,0}^{+})<0$. So we have  \eqref{eq-s-condi-m0}, and hence \eqref{eq-inequ-2} holds for $m=0$.

Now assume $r_{n,m-1}^{+}<r_{n,m}^{+}<r_{n-1,m-1}^{+}$ for $n\geq m+3$. We aim to show that
$$
r_{n,m}^{+}<r_{n,m+1}^{+}<r_{n-1,m}^{+}
$$
for $n\geq m+4$.
Clearly, $r_{n,m}^{+}<r_{n-1,m}^{+}$ for $n\geq m+4$ by \eqref{eq-mono-1}.
Therefore by the continuity of $N_{n,m}(x)$ and Theorem \ref{thm-unipzo}, it is sufficient to prove that for $n\geq m+4$,
\begin{align}\label{eq-s-condi-m}
N_{n,m+1}(r_{n,m}^{+})>0\quad {\rm and}\quad N_{n,m+1}(r_{n-1,m}^{+})<0.
\end{align}
For this purpose, let us recall the recurrence relation \eqref{eq-rec-gelNaya-mv}
\begin{align*}
N_{n,m+1}(x)
=\frac{a_{n,m}(x)}{c_{n,m}(x)}N_{n,m}(x)+\frac{b_{n,m}(x)}{c_{n,m}(x)}N_{n-1,m}(x),
\end{align*}
where
\begin{equation*}
\left\{
\begin{aligned}
&a_{n,m}(x)=(m+2-n)(m^2-n^2+4m+3)x-2n,\\[5pt]
&b_{n,m}(x)=n[(m+2-n)(m+1-n)x-2](x-1),\\[5pt]
&c_{n,m}(x)=(m+3)(m+2-n)(m+1-n)x.
\end{aligned}
\right.
\end{equation*}
It follows from \eqref{eq-rec-gelNaya-mv} that
\begin{align}\label{eq-Nnm1rnm}
N_{n,m+1}(r_{n,m}^{+})
=\frac{b_{n,m}(r_{n,m}^{+})}{c_{n,m}(r_{n,m}^{+})}N_{n-1,m}(r_{n,m}^{+}),
\end{align}
and
\begin{align}\label{eq-Nnm1n-1m}
N_{n,m+1}(r_{n-1,m}^{+})
=\frac{a_{n,m}(r_{n-1,m}^{+})}{c_{n,m}(r_{n-1,m}^{+})}N_{n,m}(r_{n-1,m}^{+}).
\end{align}

Let us first determine the sign of $N_{n,m+1}(r_{n,m}^{+})$.
By Theorem \ref{Thm-zero-bounds}, we have
$$
r_{n,m}^{+}\leq \frac{2}{(m-n)(m+1-n)}<\frac{2}{(m+2-n)(m+1-n)}
$$
for $m\geq 0$ and $n\geq m+3$.
So $(m+2-n)(m+1-n)r_{n,m}^{+}-2<0$. Since $r_{n,m}^{+}-1<0$, it follows that $b_{n,m}(r_{n,m}^{+})>0$. Clearly, $r_{n,m}^{+}>0$, and hence $c_{n,m}(r_{n,m}^{+})>0$ for $n\geq m+4$.
By Lemma \ref{lem-sign} we get $N_{n-1,m}(r_{n,m}^{+})>0$ since $0<r_{n,m}^{+}<r_{n-1,m}^{+}$ for $n\geq 4$. Thus by \eqref{eq-Nnm1rnm} we have $N_{n,m+1}(r_{n,m}^{+})>0$ for $n\geq m+4$.

It remains to prove $N_{n,m+1}(r_{n-1,m}^{+})<0$ for $n\geq m+4$.
By Theorem \ref{Thm-zero-bounds}, we have
$$
r_{n-1,m}^{+}>\frac{2n}{(m+2-n)((m+2)^2-n^2-1)}
$$
for $m\geq 0$ and $n\geq m+4$. Hence $a_{n,m}(r_{n-1,m}^{+})>0$ for $n\geq m+4$.
It is clear that $r_{n-1,m}^{+}>0$ and hence $c_{n,m}(r_{n-1,m}^{+})>0$, for $n\geq m+4$.
By Lemma \ref{lem-sign} we have $N_{n,m}(r_{n-1,m}^{+})<0$ because $r_{n,m}^{+}<r_{n-1,m}^{+}$.
Then by \eqref{eq-Nnm1n-1m} it follows that $N_{n,m+1}(r_{n-1,m}^{+})<0$ for $m\geq 0$ and $n\geq m+4$.

Since \eqref{eq-s-condi-m} has been proved, it follows that
$$
r_{n,m}^{+}<r_{n,m+1}^{+}<r_{n-1,m}^{+}
$$
for $n\geq m+4$.
By the inductive hypothesis, we have \eqref{eq-inequ-2} holds for $n\geq m+4$.
This completes the proof.
\end{proof}

\noindent{\bf Acknowledgements.}
This work was supported by the National Natural Science Foundation of China under Grant Nos. 11771330 and 11971203.

\end{document}